\documentclass[a4paper,12pt,leqno]{amsart}
\usepackage[active]{srcltx}
\usepackage{verbatim}
\usepackage{epsfig,graphicx,color,mathrsfs}
\usepackage{graphicx}
\usepackage{amsmath,amssymb,amsthm,amsfonts}
\usepackage{amssymb}
\usepackage[english]{babel}

\usepackage[left=2.7cm,right=2.7cm,top=3cm,bottom=3cm]{geometry}

\newcommand{\R}{{\mathbb R}}

\numberwithin{equation}{section}
\newtheorem{theorem}{Theorem}[section]
\newtheorem{proposition}[theorem]{Proposition}
\newtheorem{lemma}[theorem]{Lemma}

\newtheorem{remark}[theorem]{Remark}
\theoremstyle{definition}

\newcommand{\brm}{\begin{remark}\rm}
\newcommand{\erm}{\end{remark}}
\newcommand{\brms}{\begin{remark}\rm}
\newcommand{\erms}{\end{remark}}
\newcommand{\bte}{\begin{theorem}}
\newcommand{\ete}{\end{theorem}}
\newcommand{\bpr}{\begin{proposition}}
\newcommand{\epr}{\end{proposition}}
\newcommand{\ble}{\begin{lemma}}
\newcommand{\ele}{\end{lemma}}
\newcommand{\beq}{\begin{equation}}
\newcommand{\eeq}{\end{equation}}
\newcommand{\bdm}{\begin{displaymath}}
\newcommand{\edm}{\end{displaymath}}
\numberwithin{equation}{section}

\newcommand{\bos}{\begin{remark}\rm}
\newcommand{\eos}{\end{remark}}

\newcommand{\ben}{\begin{enumerate}}
\newcommand{\een}{\end{enumerate}}

\newcommand{\e }{\varepsilon }

\newcommand{\be}{\begin{equation}}
\newcommand{\ee}{\end{equation}}

\thanks{
AC is partially supported by the
Italian PRIN Research Project 2009: {\em Metodi Variazionali e Topologici
nello Studio di Fenomeni non Lineari}. \\
BS is partially supported by the
Italian PRIN Research Project 2009: {\em Metodi Variazionali e Topologici
nello Studio di Fenomeni non Lineari}, and is also partially supported by  ERC-2011-grant: \emph{Epsilon}.\\
}
\begin{document}

\title[Uniqueness]{A uniqueness result  for some singular semilinear elliptic equations}
\author {Annamaria Canino* and Berardino Sciunzi*}

\date{\today}

\address{* Dipartimento di Matematica e Informatica, UNICAL,
Ponte Pietro  Bucci 31B, 87036 Arcavacata di Rende, Cosenza, Italy.}

\email{canino@mat.unical.it, sciunzi@mat.unical.it}

\keywords{Singular semilinear elliptic equations, Comparison principles, Uniqueness of the solutions}

\subjclass[2010]{35J75,35B51,35B06}


\begin{abstract}
Given $\Omega$ a bounded open subset   of $\R^N$,
we consider nonnegative solutions to the singular semilinear elliptic equation $-\Delta\,u\,=\,\frac{f}{u^{\beta}}$ in $H^1_{loc}(\Omega)$,  under zero Dirichlet boundary conditions.  For  $\beta>0$ and  $f\in L^1(\Omega)$, we prove that the solution is unique.
\end{abstract}

\maketitle

\date{\today}



\maketitle

\section{introduction}

Let $\beta >0$ and and let $\Omega$ be a bounded open subset of $ \mathbb{R}^n$. We consider  $u\in H^1_{loc}(\Omega)$ weak solution to:
\begin{equation}\label{problemjumping}
\begin{cases}\displaystyle
-\Delta\,u\,=\,\frac{f}{u^{\beta}}  & \text{in $\Omega$,}  \\
u> 0 & \text{in $\Omega$,}  \\
\end{cases}
\end{equation}
where we assume that $f\in L^1(\Omega)$ and the assumption $u>0$ in $\Omega$ means that, for any compact set $K\subset \Omega$, we have that $$\underset{K}{\text{ess inf}}\, u>0.$$ The equation in \eqref{problemjumping} has to be understood in the weak distributional meaning, namely:
\begin{equation}\label{eq:fgakghfhksajdgfja}
\int_\Omega\,\nabla u\,  \nabla\varphi\,dx  = \int_\Omega \frac{f}{u^\beta}\varphi\,dx  \qquad\forall \varphi\in C^1_c(\Omega)\,.
\end{equation}
We study the uniqueness of the solution, when  zero Dirichlet boundary conditions are imposed, according to the following:
\begin{remark}\label{rediri}
Since, in general, the solution $u$ is not continuous up to the boundary and is not in $H^1_0(\Omega)$, we need to specify the meaning of the Dirichlet boundary condition.
In fact, following \cite{CanDeg}, we say
that $u\leqslant 0$ on $\partial\Omega$ if, for every $\e>0$, it follows that
\[
(u-\e)^+\in H^1_0(\Omega)\,.
\]
We will say that $u= 0$ on $\partial\Omega$ if $u$ is nonnegative and $u\leqslant 0$ on $\partial\Omega$.
\end{remark}
\noindent The study of singular semilinear problems was started in the pioneering work \cite{crandall} and
it is worth mentioning the contributions in \cite{boccardo,C23,CanDeg,CGS,gatica,saccon,kaw,lair,LM,stuart}. One of the main difficulties in this issue is given by the fact that, in general, the solution  is not in $H^1_0(\Omega)$. In particular it has been shown in \cite{LM}
that the solution cannot belongs to $H^1_0(\Omega)$ if $\beta\geqslant 3$.\\

The existence of a solution in our case, namely considering  \eqref{problemjumping} and imposing zero Dirchlet boundary conditions according to Remark \ref{rediri}, follows by the results in \cite{boccardo} where 
 $f$ is a nonnegative function such that
$f \in L^1(\Omega)$ if $\beta \geq 1$ while, if $0<\beta<1$, the further assumption
$f \in L^m(\Omega)$  with $m = \frac{2N}{N + 2 + \beta(N-2)}$ is required.
The solution found in \cite{boccardo} is obtained as the limit of a sequence $u_n$ given by the solutions to the truncated regularized problem:
\begin{equation}\label{problemjumpingdkshfs}
\begin{cases}\displaystyle
-\Delta\,u_n\,=\,\frac{f_n}{\Big(u_n+\frac{1}{n}\Big)^{\beta}}  & \text{in $\Omega$,}  \\
u= 0 & \text{on $\partial\Omega$,}\,.  \\
\end{cases}
\end{equation}
It is proved in \cite{boccardo} that the approximating solutions $u_n$ are uniformly bounded away from zero in the interior of the domain, and this allows to pass to the limit thus proving the existence of a solution to \eqref{problemjumping} such that:
\begin{equation}\label{fjnbklwqhrHL}
u\in H^1_{loc}(\Omega) \quad\text{and }\quad u^q\in H^1_0(\Omega)\quad\text{for}\quad q\,:=\,\max\{1,\frac{\beta+1}{2}\}\,.
\end{equation}
Note that in this case the solution is
strictly bounded away from zero on compact sets of $\Omega$, by construction.\\

It follows that the solution found in \cite{boccardo} has zero Dirichlet boundary condition in the meaning of
Remark \ref{rediri}. This is obvious if $\beta\leqslant 1$. If else $\beta>1$ this follows exploiting the fact that
$u^q\in H^1_0(\Omega)$ for $q\,:=\,\max\{1,\frac{\beta+1}{2}\}$ and observing that $u\geqslant \e$ on the support of
$(u-\e)^+$. A detailed proof of this fact will be provided in the proof of  Theorem \ref{UcasoB}.\\

\noindent Even though the uniqueness of the solution is expected because of the fact that the nonlinearity is decreasing with respect to
the variable $u$, there are no general result in the literature  under our general assumptions. This is manly caused by the lack of regularity of the solutions up to the boundary. We refer to \cite{CanDeg,crandall} for the case
$f=1$.\\

\noindent On the other hand, it is  standard to prove the uniqueness of the solution  when
nondecreasing locally Lipschitz continuous (non singular) nonlinearities are considered. It is also not hard to prove a uniqueness result, for singular elliptic equations, in the space $H^1_0(\Omega)$. This is quite a well known result, anyway we will provide a short and simple proof for the reader's convenience in Theorem \ref{ordinaryuniqueness}.
Let us emphasize the fact that, by  Theorem \ref{ordinaryuniqueness}, it follows  the uniqueness of the solution to \eqref{problemjumping} in the case $0<\beta<1$. In this case case in fact, assuming that $f \in L^m(\Omega)$  with $m = \frac{2N}{N + 2 + \beta(N-2)}$, the solutions found in \cite{boccardo} are in $H^1_0(\Omega)$
and this is a space naturally associated to the problem.\\

\noindent The crucial point here is the fact that, as already remarked, the solutions in general are not in $H^1_0(\Omega)$. This is the motivation for which, in the case $\beta>1$, the uniqueness of the solution has not been already proved.\\

\noindent Our main result is the following:
\begin{theorem}\label{main}
Let $\beta>1$ and let $f\in L^1(\Omega)$ be  non-negative. Then, under zero Dirichlet boundary conditions, the solution to \eqref{problemjumping} is unique.
\end{theorem}
We will prove Theorem \ref{main} exploiting the variational approach introduced in \cite{CanDeg}. Actually
Theorem \ref{main} will be a consequence of a more general \emph{weak comparison principle}, see Theorem \ref{comparison}.\\

\noindent As a corollary of our result, we will obtain the proof that the solution found in \cite{boccardo}
is the only one solution in the class of solutions fulfilling \eqref{fjnbklwqhrHL}. Furthermore, such a solution also
fulfils the Dirichlet boundary condition in the meaning of
Remark \ref{rediri}, and is unique also in this class. More precisely we will prove the following:

\begin{theorem}\label{UcasoB}
The solution to \eqref{problemjumping} is unique in the class of functions fulfilling \eqref{fjnbklwqhrHL}. More precisely, for $\beta>1$ and  $f\in L^1(\Omega)$ non-negative, if $u$ and $v$ are two solutions of  \eqref{problemjumping} fulfilling \eqref{fjnbklwqhrHL}, then it follows that $u\equiv v$ a.e. in  $\Omega$.
\end{theorem}
Combining Theorem \ref{main} and Theorem \ref{UcasoB}, we will have the quoted uniqueness result.\\

\noindent We will conclude the paper pointing out a first simple consequence of our uniqueness result. It follows in fact by uniqueness that, if the domain $\Omega$ and the datum $f$ have some symmetry, then the solution inherits the symmetry. This will be precisely stated and proved in Theorem \ref{Symmetry}. It follows in particular that, if the domain is a ball or an annulus, then the solution is unique and radial.\\

The paper is organized as follows: In Section \ref{kduhfkuer} we prove Theorem \ref{comparison} and then we exploit it to prove Theorem \ref{main}. As a consequence we deduce Theorem \ref{UcasoB} and Theorem \ref{Symmetry}.
In Section \ref{dfssfjkgdfsg} we prove  Theorem \ref{ordinaryuniqueness}.

\section{Proof of  Theorem \ref{main}}\label{kduhfkuer}
Let us start defining the real valued function $g_{k}(s)$ by
\begin{equation}\nonumber
g_{k}(s)\,:=\,
\begin{cases}
\max\{-s^{-\beta}\,,\,-k\}\quad\text{if}\quad s>0,\\
-k\qquad\qquad\quad\qquad\text{if}\quad s\leqslant 0\,.\\
\end{cases}
\end{equation}
Then we consider the real valued function $\Phi_{k}(s)$ defined by the conditions
\begin{equation}\nonumber
\begin{cases}
\Phi_{k}'(s)\,=\,g_{k}(s),\\
\Phi_{k}(1)=0\,,\\
\end{cases}
\end{equation}
namely we consider the primitive of $g_{k}(s)$ that is equal to zero for $s=1$.\\

Let us consequently consider the functional $J_k\,:H^1_0(\Omega)\rightarrow [-\infty\,,\,+\infty]$ defined by
\begin{equation}\nonumber
J_k(\varphi)\,:=\,\frac{1}{2} \int_\Omega\,|\nabla\,\varphi|^2\,dx+\int_\Omega f\cdot\Phi_{k}(\varphi)\,dx\qquad\varphi\in H^1_0(\Omega)\,.
\end{equation}
In the following we will exploit the fact that $J_k(\varphi)\in\mathbb{R}$ if we assume that $\varphi$ is in the positive cone.
Let $w$ be defined as the minimum of $J_k$ on the convex set
\[
\mathcal K\,:=\,\{\varphi\in H^1_0(\Omega)\,:\, 0\leqslant \varphi\leqslant v\,\,\text{a.e. in}\,\,\Omega\}\,.
\]
By \cite{stampacchia} it follows that
\begin{equation}\nonumber
 \int_\Omega\,\nabla w\nabla (\psi-w)\,dx\geqslant -\int_\Omega\,f\cdot \Phi_{k}'(w)(\psi-w)\quad \text{for}\,\,
 \psi\in w+\left(H^1_0(\Omega)\cap L^\infty_c(\Omega)\right)\,\,\text{and}\,\,0\leqslant \psi\leqslant v\,.
\end{equation}
\begin{lemma}\label{lemmause}
We have that
\begin{equation}\label{ksjfskgfdfjigf}
 \int_\Omega\,\nabla w\nabla \psi\,dx\geqslant -\int_\Omega\,f\cdot \Phi_{k}'(w)\psi\quad \text{for}\,\,
 \psi\in C^\infty_c(\Omega)\qquad\text{with}\quad \psi\geqslant 0\,\,\text{a.e. in}\,\,\Omega\,.
\end{equation}
\end{lemma}
\begin{proof}
To prove this let us consider $\theta\in C^\infty_c(\mathbb{R})$ with $0\leqslant \theta\leqslant 1$ for $t\in\mathbb{R}$, $\theta (t)=1$
for $t\in[-1,1]$ and $\theta (t)=0$
for $t\in(-\infty,-2]\cup[2,\infty)$. Then, for any $\omega\in C^\infty_c(\Omega)$ with $\omega\geqslant 0$ in $\Omega$,
we set
\[
\omega_k\,:=\,\theta(\frac{w}{k})\,\omega,\qquad\quad \omega_{k,t}\,:=\,\min\{w+t\omega_k\,,\,v\}\,,
\]
with $k\geqslant 1$ and $t>0$.
We have that
$\omega_{k,t}\in w+\left(H^1_0(\Omega)\cap L^\infty_c(\Omega)\right)$ and $w\leqslant \omega_{k,t}\leqslant v$, so that

\begin{equation}\nonumber
 \int_\Omega\,\nabla w\nabla (\omega_{k,t}-w)\,dx\geqslant -\int_\Omega\,f\cdot \Phi_{k}'(w)(\omega_{k,t}-w)\,.
\end{equation}
Consequently
\begin{equation}\nonumber
\begin{split}
 &\int_\Omega\,|\nabla (\omega_{k,t}-w)|^2
 +f\cdot(\Phi_{k}'(\omega_{k,t})-\Phi_{k}'(w))(\omega_{k,t}-w)\,dx\\
 &\leqslant
 \int_\Omega\,\nabla \omega_{k,t}\nabla (\omega_{k,t}-w)\,+\,
 f\cdot\Phi_{k}'(\omega_{k,t})(\omega_{k,t}-w)\,dx\\
 &=\int_\Omega\,\nabla \omega_{k,t}\nabla (\omega_{k,t}-w-t\omega_k)\,+\,
 f\cdot\Phi_{k}'(\omega_{k,t})(\omega_{k,t}-w-t\omega_k)\,dx\\
 &+t\int_\Omega \nabla \omega_{k,t}\nabla \omega_{k}+f\cdot\Phi_{k}'(\omega_{k,t})\omega_k\,dx\\
 &=\int_\Omega\,\nabla v\nabla (\omega_{k,t}-w-t\omega_k)\,+\,
 f\cdot\Phi_{k}'(v)(\omega_{k,t}-w-t\omega_k)\,dx\\
 &+t\int_\Omega \nabla \omega_{k,t}\nabla \omega_{k}+f\cdot\Phi_{k}'(\omega_{k,t})\omega_k\,dx\,.
 \end{split}
\end{equation}
Note now that, by the definition of $\Phi_{k}$, it follows that $v$ is also a supersolution to the equation
$-\Delta z=-\Phi_{k}'(z)$, so that, observing that $\omega_{k,t}-w- t\omega_k\leqslant 0$, we deduce
\begin{equation}\nonumber
\begin{split}
 &\int_\Omega\,|\nabla (\omega_{k,t}-w)|^2
 +f\cdot(\Phi_{k}'(\omega_{k,t})-\Phi_{k}'(w))(\omega_{k,t}-w)\,dx\\
 &\leqslant
 t\int_\Omega \nabla \omega_{k,t}\nabla \omega_{k}+f\cdot\Phi_{k}'(\omega_{k,t})\omega_k\,dx\,.
 \end{split}
\end{equation}
Exploiting again the fact  that $\omega_{k,t}-w\leqslant t\omega_k$, by simple computations we deduce that
\begin{equation}\nonumber
\begin{split}
\int_\Omega \nabla \omega_{k,t}\nabla \omega_{k}+f\cdot\Phi_{k}'(\omega_{k,t})\omega_k\,dx\geqslant -\int_\Omega
 f\cdot|\Phi_{k}'(\omega_{k,t})-\Phi_{k}'(w)||\omega_{k}|\,dx\\
 \end{split}
\end{equation}
We can pass to the limit for $t\rightarrow 0$ exploiting also the Lebesgue Theorem obtaining
\begin{equation}\nonumber
\begin{split}
\int_\Omega \nabla w\nabla \omega_{k}+f\cdot\Phi_{k}'(w)\omega_k\,dx\geqslant 0\,.
 \end{split}
\end{equation}
The claim, namely the proof of \eqref{ksjfskgfdfjigf}, follows letting $k$ tend to infinity.\\
\end{proof}

Now we are in position to prove our \emph{weak comparison principle}, namely we have the following:
\begin{theorem}\label{comparison}
Let $\beta>1$ and let $f\in L^1(\Omega)$ be  non-negative. Let $u\in H^1_{loc}(\Omega)$ be a subsolution to  \eqref{problemjumping} such that $u\leq 0$ on $\partial\Omega$ and let $v\in H^1_{loc}(\Omega)$ be a supersolution to
\eqref{problemjumping}. Then, $u\leqslant v$ a.e. in $\Omega$.
\end{theorem}

\begin{proof}
We start noticing that, since $w\in H^1_0(\Omega)$ with $w\geqslant 0$ a.e. in $\Omega$, exploiting also the fact that $u\leqslant 0$  on $\partial\Omega$ according to Remark \ref{rediri},
it follows  that
\[
(u-w-\e)^+\in H^1_0(\Omega)\,.
\]
Therefore, by \eqref{ksjfskgfdfjigf} and standard density arguments, it follows
\begin{equation}\label{eq111}
 \int_\Omega\,\nabla w\nabla T_\tau\left((u-w-\e)^+\right)\,dx\geqslant -\int_\Omega\,f\cdot \Phi_{k}'(w)T_\tau\left((u-w-\e)^+\right)\,dx\,
\end{equation}
for $T_\tau (s)\,:=\, \min\{s,\tau\}$ for $s\geqslant 0$ and $T_\tau (-s)\,:=\,-T_\tau (s)$ for $s< 0$.\\
Let now $\varphi_n\in C^\infty_c(\Omega)$ such that $\varphi_n\rightarrow (u-w-\e)^+$ in $H^1_0(\Omega)$ and set
\[
\tilde\varphi_{\tau,n}\,:=\,T_\tau (\min\{(u-w-\e)^+,\varphi_n^+\})\,.
\]
It follows that $\tilde\varphi_{\tau,n}\in H^1_0(\Omega)\cap L^\infty_c(\Omega)$ so that

\begin{equation}\nonumber
 \int_\Omega\,\nabla u\nabla \tilde\varphi_{\tau,n}\,dx\leqslant \int_\Omega\,\frac{f}{u^\beta}\,\tilde\varphi_{\tau,n}\,dx\,.
\end{equation}
Passing to the limit as $n$ tends to infinity, it is easy to deduce that
\begin{equation}\label{eq222}
 \int_\Omega\,\nabla u\nabla T_\tau\left((u-w-\e)^+\right)\,dx\leqslant \int_\Omega\frac{f}{u^\beta}T_\tau\left((u-w-\e)^+\right)\,dx\,.
\end{equation}
Now we take $\e>0$ such that $\e^{-\beta}<k$ and, by \eqref{eq111} and \eqref{eq222}, we deduce
\begin{equation}\nonumber
\begin{split}
 \int_\Omega\,|\nabla T_\tau\left((u-w-\e)^+\right)|^2\,dx
 &\leqslant
 \int_\Omega\,f \cdot\left(\frac{1}{u^\beta}+\Phi_{k}'(w)\right)
 T_\tau\left((u-w-\e)^+\right)\,dx\\
 &\leqslant
 \int_\Omega\,f \cdot\left(-\Phi_{k}'(u)+\Phi_{k}'(w)\right)
 T_\tau\left((u-w-\e)^+\right)\,dx\\
 &\leqslant 0\,.
 \end{split}
\end{equation}
By the arbitrariness  of $\tau>0$  we deduce that
\[
u\leqslant w+\e \leqslant v+\e\qquad \text{a.e. in}\quad \Omega
\]
and the thesis follows letting $\e\rightarrow 0$.

\end{proof}

As direct consequence of Theorem \ref{comparison} we will obtain the proof of Theorem \ref{main}. 
\begin{proof}[\underline{Proof of  Theorem \ref{main}}]
 If $u$ and $v$
are two solutions to \eqref{problemjumping} with zero Dirichlet boundary condition, then we have that
$u\leqslant v$ by Theorem \ref{comparison}. In the same way it follows that $v\leqslant u$, and the proof is done.
\end{proof}
Moreover, by Theorem \ref{main}, we can prove Theorem \ref{UcasoB}.
\begin{proof}[\underline{Proof of   Theorem \ref{UcasoB}}]
 We first claim that, if $u$ fulfills \eqref{fjnbklwqhrHL}, then, for every $\varepsilon>0$, it follows that
\[
(u-\e)^+\in H^1_0(\Omega)\,.
\]
This is obvious if $0<\beta\leqslant 1$. For $\beta>1$ let $\varphi_n\in C^1_c(\Omega)$ such that $\varphi_n$ converges to $u^{\frac{\beta+1}{2}}$ in $H^1_0(\Omega)$ and set
\[
\psi_n \,:=\, (\varphi_n^{\frac{2}{\beta+1}}-\e)^+\,.
\]
It follows that $\psi_n$ is uniformly bounded in $H^1_0(\Omega)$ and it converges a.e. to $(u-\e)^+$. Therefore we obtain that $(u-\e)^+\in H^1_0(\Omega)$ and the claim is proved. The proof of Theorem \ref{UcasoB} follows now  as a consequence of Theorem \ref{main}.
\end{proof}

We now deduce from the uniqueness of the solution a symmetry result. We have the following:
\begin{theorem}\label{Symmetry}
Let $u\in H^1_{loc}(\Omega)$ be the solution to \eqref{problemjumping} under zero Dirichlet boundary condition. Assume that the domain $\Omega$ is symmetric with respect to some hyperplane $T_\lambda^\nu\,:=\,\{x\cdot\nu=\lambda\}$, $\lambda\in\R$ and $\nu\in S^{N-1}$. Then, if $f$ is symmetric with respect to the hyperplane $T_\lambda^\nu$, then $u$ is symmetric with respect to the hyperplane $T_\lambda^\nu$ too. In particular, by rotation and translation invariance, we can consider the case $T_\lambda^\nu=T_0^{e_1}=\{x_1=0\}$. It follows in this case that, if $\Omega$ is symmetric in the $x_1$-direction and $f(x_1,x')=f(-x_1,x')$ (with $x'\in\R^{N-1}$), then
\[
u(x_1,x')=u(-x_1,x')\,\qquad\text{a.e. in}\quad \Omega.
\]
In particular, if $\Omega$ is a ball or an annulus (centered at the origin) and $f$ is radially symmetric, then $u$ is radially symmetric.
\end{theorem}
\begin{proof}
By rotation and translation invariance, we may and we will assume that
$\Omega$ is symmetric in the $x_1$-direction and $f(x_1,x')=f(-x_1,x')$ (with $x'\in\R^{N-1}$).
Setting
\[
v(x_1,x')\,:=\, u(-x_1,x')\,,
\]
it follows that $v$ is a solution to \eqref{problemjumping} with zero Dirichlet boundary condition. By uniquenes, namely applying Theorem \ref{main}, it follows that $u=v$, that is
\[
u(x_1,x')=u(-x_1,x')\,\qquad\text{a.e. in}\quad \Omega\,,
\]
ending the proof.
\end{proof}

\section{Uniqueness in $H^1_0(\Omega)$}\label{dfssfjkgdfsg}
For the readers convenience we provide here a simple proof of the uniqueness of the solution in $H^1_0(\Omega)$.
We have the following:
\begin{theorem}\label{ordinaryuniqueness}
Let $\beta>0$ and let $f\in L^1(\Omega)$ be  non-negative. Then, under zero Dirichlet boundary conditions, the solution to \eqref{problemjumping}  in $H^1_0(\Omega)$ (if it exists) is unique in $H^1_0(\Omega)$.
\end{theorem}
\begin{proof}
Let $\beta>0$ and let $f\in L^1(\Omega)$ be  non-negative and  consider $u,v\in H^1_0(\Omega)$ two solutions
to \eqref{problemjumping}. Let us show that $u\leqslant v$. \\

\noindent Observe that $(u-v)^+\in H^1_0(\Omega)$ and consider $\varphi_n\in C^\infty_c(\Omega)$ such that $\varphi_n$
converges to $(u-v)^+$ in $ H^1_0(\Omega)$. Set now
\[
\tilde\varphi_n\,:=\,\min\{(u-v)^+\,,\,\varphi_n^+\}\,.
\]
It is easy to verify that $\tilde\varphi_n\in H^1_0(\Omega)$ and has compact support in $\Omega$. Therefore, by using it as test function, we obtain:
\begin{equation}\nonumber
\begin{split}
 \int_\Omega\,
 \nabla (u-v)\nabla \tilde\varphi_n
 \,dx&=
 \int_\Omega\,f \cdot\left(\frac{1}{u^\beta}-\frac{1}{v^\beta}\right)
 \tilde\varphi_n\,dx\leqslant 0\,.
 \end{split}
\end{equation}
Passing to the limit, we have
\begin{equation}\nonumber
 \int_\Omega\,
 |\nabla (u-v)^+|^2
 \,dx\leqslant 0\,,
\end{equation}
that implies $u\leqslant v$ in $\Omega$. In the same way we can prove that $v\leqslant u$ in $\Omega$, namely $u= v$ in $\Omega$.
\end{proof}

\bigskip

\end{document}